\documentclass[11pt]{amsart}
\usepackage{amsmath,amsthm,amsfonts,amscd,amssymb,eucal,latexsym, fullpage}
\usepackage{mathrsfs}
\usepackage[numbers]{natbib}
\usepackage[fit]{truncate}
\usepackage{fullpage}
\usepackage{graphicx}
\usepackage{epsfig}
\usepackage{color,soul}
\usepackage[all,cmtip]{xy}
\usepackage[top=1.125in, bottom=1.125in, left=1.25in, right=1.25in]{geometry}
\newcommand{\truncateit}[1]{\truncate{0.8\textwidth}{#1}}
\newcommand{\scititle}[1]{\title[\truncateit{#1}]{#1}}

\theoremstyle{plain}
\newtheorem{theorem}{Theorem}[section]
\newtheorem{corollary}[theorem]{Corollary}
\newtheorem{lemma}[theorem]{Lemma}

\theoremstyle{definition}
\newtheorem{definition}[theorem]{Definition}
\newtheorem{remark}[theorem]{Remark}

\def\de{\delta} \def\dl{\partial}   \def\Si{\Sigma} \def\si{\sigma}
\def\su{\subset}    \def\De{\Delta} 
\def\al{\alpha} \def\be{\beta}      \def\ga{\gamma} \def\Ga{\Gamma}  \def\ra{\rightarrow} 
\def\ti{\tilde}     \def\z{\times}  \def\x{\cdot}
   \def\ep{\varepsilon}\def\F{\Phi}    

\newcommand{\R}{\mathbb{R}}

\newcommand{\Z}{\mathbb{Z}}
\newcommand{\Mm}{\mathcal{M}}
\newcommand{\Bb}{\mathcal{B}}
\newcommand{\Nn}{\mathcal{N}}
\newcommand{\Zz}{\mathcal{Z}}
\newcommand{\Cc}{\mathcal{C}}
\newcommand{\Ll}{\mathcal{L}}
\newcommand{\Tt}{\mathcal{T}}

\newcommand{\Nc}{\mathfrak{N}}

\DeclareMathOperator{\mass}{mass}
\DeclareMathOperator{\thick}{thick}
\DeclareMathOperator{\diam}{diam}
\DeclareMathOperator{\dist}{dist}
\DeclareMathOperator{\length}{length}

\DeclareMathOperator{\vol}{vol}

\begin{document}
\begin{abstract}
In this paper, we prove that if $M$ is a closed 4-dimensional Riemannian manifold with trivial first homology group, Ricci curvature  $|Ric|\leq3$, diameter  $\diam(M)\leq D$ and volume  $\vol(M)>v>0$, then the area of a smallest 2-dimensional stationary integral varifold in $M$ is bounded by $F(v,D)$, for some function $F$ that only depends on $v$ and $D$. Our bound for the area is based on the estimation of the first homological filling function of $M$.
\end{abstract}

\scititle{An Upper bound for the smallest area of a minimal surface in manifolds of dimension four}
\author{Nan Wu, Zhifei Zhu}
\date{\today}
\maketitle

\section{Introduction}
In this paper, we prove the following result.
\begin{theorem}\label{thm}
Let $M$ be a closed 4-dimensional Riemannian manifold with Ricci curvature $|Ric|\leq 3$, volume $\vol(M)>v>0$, and diameter $\diam(M)\leq D$. Suppose that the first homology group $H_1(M)$ is trivial, then the area $A(M)$ of a smallest 2-dimensional stationary integral varifold in $M$ satisfies
$$A(M) \leq F(v,D),$$
for some function $F$ that only depends on $v$ and $D$.
\end{theorem}

The k-dimensional stationary integral varifold is a generalization of the k-dimensional minimal submanifold. In \cite{pitts2014existence}, J.Pitts proved the existence of the $k-$stationary integral varifold in a closed Riemannian manifold $M^n$. Specifically, he showed that one could get a non-trivial $k-$stationary integral varifold for all $k \leq i$ whenever $H_i(M^n)$ is non-trivial. Pitts also proved the smoothness of such stationary integral varifold in $M^n$ for $k=n-1$ and $3 \leq n \leq 5$. The existence and regularity results of Pitts were extended by R.Schoen and L.Simon in \cite{schoen1981regularity} to $k=n-1$ and $n \leq 7$. 

In \cite{gromov1983filling}, M.~Gromov asked whether the length of a shortest periodic geodesic in a $n$-dimensional Riemannian manifold $M^n$ can be bounded by $c(n)\vol (M)^{1/n}$. A generalization of this question is whether one can relate the volume of a $k$-dimensional stationary integral varifold to the parameters of the ambient manifold $M^n$. Various results have been obtained in this direction. In \cite{nabutovsky2004volume}, A.~Nabutovsky and R.~Rotman showed the existence of a strongly stationary 1-cycle in a simply connected closed Riemannian manifold $M^n$ whose 1-mass bounded by the diameter of $M^n$. In \cite{glynn2017width}, P.~Glynn-Adey and Ye.~Liokumovich proved an upper bound of the area of a smallest $n-1$-dimensional stationary integral varifold in a closed Riemannian manifold $M^n$ in terms of Ricci curvature lower bound and volume upper bound of $M^n$. Our result is the first one that relates the volume of a $2$-dimensional stationary integral varifold to the Ricci curvature, diameter and volume of the ambient manifold $M$ of dimension $4$.

The proof of Theorem \ref{thm} relies on the relation between the area of $2$-dimensional stationary integral varifold and the first homological filling function $HF_1$ of the ambient manifold $M$ that is proven by A.~Nabutovsky and R.~Rotman. In \cite{nabutovsky2006curvature}, Nabutovsky and Rotman proved the following effective version of the results in \cite{pitts2014existence}.
\begin{theorem}\label{thnr}
Let $M$ be a n-dimensional closed Riemannian manifold with diameter $D$ and the trivial first homology group. Then the area of a $2$-dimensional stationary integral varifold in $M$ satisfies
$$A(M)\leq \frac{(n+1)!}{2}HF_1(2D).$$
\end{theorem}

Homological filling function is extensively studied and applied in geometric group theory and Riemannian geometry (cf. \cite{gromov1983filling}). It provides the largest mass of a 2-chain with integer coefficients that fills a 1-cycle with given 1-mass upper bound. Before we give the formal definition, let us first introduce some basic concepts in geometric measure theory, \cite{federer2014geometric}.
\begin{definition}\label{def2}
Let $M$ be a Riemannian manifold. The space of singular Lipschitz $k$-chains in $M$ with coefficient group $\Z$ is $\Cc_k(M;\Z)=\{\sum_i a_if_i|f_i:\De^k\ra M \text{ a Lipschitz map}, a_i\in \Z \}$. Let $\dl_k:\Cc_k\ra \Cc_{k-1}$ be the boundary map. We define the space of $k$-cycles to be $\Zz_k(M;G)=\ker \dl_k$.
\end{definition}

We define the mass of a singular Lipschitz $k$-chain to be the following.

\begin{definition}\label{def3}
Let $C=\sum_i a_if_i\in \Cc_k(M;\Z)$. We define the $k$-mass of $C$ to be:
$$\mass_k (C)=\sum_i |a_i|\int_{\De_k}f_i^*(d \vol_g),$$
where $d \vol_g$ is the Riemannian volume form on $M$.
\end{definition}

Now we define the first homological filling function of a Riemannian manifold $M$ as the following.

\begin{definition}
Let $M$ be a closed Riemannian manifold with trivial first homology group. The first homological filling function $HF_1:\R_{\geq0}\ra \R_{\geq 0}$ is defined by
$$HF_1(l)=\sup_{\mass_1(Z)\leq l}\left(\inf_{\dl C=Z} \mass_2(C)\right),$$
where the supremum is taken over all singular Lipschitz $1$-cycle $Z\in \Zz_1(M;\Z)$ with mass bounded by $l$.
\end{definition}

In this paper, we are going to majorize the first homological filling function of a 4-dimensional Riemannian manifold $M$ by its Ricci curvature, volume and diameter. More precisely, we prove that:
\begin{theorem}\label{thm1}
Let $M$ be a closed 4-dimensional Riemannian manifold with Ricci curvature $|Ric|\leq 3$, volume $\vol(M)>v>0$, and diameter $\diam(M)\leq D$. Suppose that the first homology group $H_1(M)$ is trivial, then every Lipschitz singular 1-cycle $C\in \Zz_1(M,\Z)$ bounds a singular 2-chain in $M$ whose 2-mass is bounded by $f_1(v,D)\cdot \mass_1(C)+f_2(v,D)$, where $f_1$ and $f_2$ are functions that only depend on $v$ and $D$.

In particular, the first homological filling function $HF_1$ of $M$ satisfies
$$HF_1(l)\leq f_1(v,D)l+f_2(v,D).$$
\end{theorem}

\begin{remark}
Combined with Theorem~\ref{thnr}, Theorem~\ref{thm1} immediately implies our main result Theorem~\ref{thm}, where $F(v,D)=120f_1(v,D)D+60f_2(v,D)$.
\end{remark}

\begin{remark}
Note that the functions $f_1$ and $f_2$ in the above Theorem~\ref{thm1} are not explicit in general. In the proof of the above theorem, we used the ``bubble tree'' construction in the work of J.~Cheeger and A.~Naber  \cite{cheeger2014regularity}. This construction depends on ``$\ep$-regularity'' Theorem \cite[Theorem 2.11]{cheeger2014regularity}. If one can write down an explicit expression for the constants in the ``$\ep$-regularity'' Theorem, then we can write down a formula for the functions $f_1$ and $f_2$. In particular, if $M$ is an Einstein manifold, then by the work \cite{anderson1989ricci} and \cite{anderson1992thel} of M.~Anderson, one can estimate $\ep \sim C\x v^{-1/2}$.

Note that the work of Cheeger and Naber uses the previous theory of manifolds with Ricci curvature bounded below developed by J.~Cheeger and T.~Colding \cite{colding1997ricci}, \cite{cheeger1996lower}, \cite{cheeger1997structure} and work of M.~Anderson \cite{anderson1989ricci}, \cite{anderson1992thel} and Cheeger and Anderson \cite{anderson1991diffeomorphism}.
\end{remark}

\begin{remark}
The second term $f_2(v,D)$ in the above estimation may not be eliminated. One can consider the following motivating example. We consider a sequence of Riemannian manifolds with boundary $\{(M_i,\dl M_i)\}$ where each $M_i$ is diffeomorphic to the disk bundle over $\R P^2$, equipped with the Eguichi-Hanson metric
$$ds^2=\left(1-\frac{a_i}{r^4}\right)^{-1}dr^2+\frac{r^2}{4}\left(1-\frac{a_i}{r^4}\right)\si_3^2+\frac{r^2}{4}(\si_1^2+\si_2^2),$$
where $\si_1, \si_2$ and $\si_3$ are the 1-forms on $S^3/\Z^2$, and $a_i\ra 0$. Note that $M_i$ is Ricci flat. The boundary of $M_i$ is a lens space, hence is orientable. One may fill the boundary $\dl M_i$ by a simply-connected 4-manifold and the resulting manifold will have trivial first homology group (cf. \cite{lickorish1962representation}). In this case, we consider a generator of the fundamental group of the base space $\R P^2$. The length of this generator $l_i\ra 0$ as $a_i\ra 0$. However, this circle does not bound any 2-chain in the interior of $M_i$. Therefore, a 2-chain that fills this circle must leave the boundary $\dl M_i$ and the size of this 2-chain is proportional to the diameter of the manifold regardless how small $l_i$ is.
\end{remark}

 Let us describe the idea of the proof of the above Theorem~\ref{thm1}. We denote by $\Mm_1(4,v,D)$, the set of closed 4-dimensional Riemannian manifolds with $H_1(M)=0$, Ricci curvature $|Ric|\leq 3$, volume $\vol(M)>v>0$, and diameter $\diam(M)\leq D$.

Suppose that $M$ is a manifold in $\Mm_1(4,v,D)$ and $C\in \Zz_1(M;\Z)$ a 1-cycle. We would like to construct a filling of the cycle $C$ whose 2-mass is bounded in terms of $\mass_1(C)$, $v$, and $D$. The construction of this filling is based on the results proved by A.~Nabutovsky and R.~Rotman in \cite{nabutovsky2003upper} and J.~Cheeger and A~.Naber in \cite{cheeger2014regularity}.

Let us begin with a simple case where the manifold $M$ is covered by some contractible metric balls of radius $r$, where $r$ is a fixed constant that is less than the harmonic radius of $M$. We further assume that the the number of the balls in this covering is $N$. In this case, following from the methods developed in \cite{nabutovsky2003upper}, one can show that the cycle $C$ bounds a 2-chain in $M$ with $2$-mass bounded by $F(N)D\cdot \mass_1(C)$. The idea of the proof in this case is the following. Let $\Nc$ be the nerve of the covering of $M$. We consider a geodesic graph $\Si$ which is obtained by connecting the center of the intersecting balls in the covering by minimizing geodesics. There is a natural map from $\Si$ to the one-skeleton $\Nc^{(1)}$ of the nerve. Given a 1-cycle $C$ with $\mass_1(C)=L$. We first look for a 1-cycle $C_1$ in $\Si$ which is homologous to $C$ such that $C-C_1$ bounds a 2-chain in $M$ with mass $c\cdot Lr$ for some constant $c$. We then map the 1-cycle $C_1$ to the nerve $\Nc$ to get a cycle $C_2$ with simplicial length $L/r$. The estimations in \cite{nabutovsky2003upper} show that one can always fill this cycle $C_2$ in $\Nc$ with no more than $F(N)\cdot L/r$ 2-simplices, where $F$ is a function that only depends on the number of the balls $N$ in the covering. Now each simplex in this case essentially corresponds to a filling of a geodesic triangle in $\Si$ whose area is $O(r^2)$. Therefore in this case, one can fill $C_1$ by a 2-chain in $M$ with mass bounded by $F(N)\cdot L/r\cdot r^2=F(N)Lr\leq F(N)LD$.

However, in general, it is difficult to cover the manifold $M$ by a certain amount of the contractible metric balls of definite radius. In the case where $M\in \Mm_1(4,v,D)$, Cheeger and Naber proved in \cite{cheeger2014regularity} that one can cover the manifold by $N(v,D)$ harmonic balls but the radius of these metric balls is not uniformly bounded from below. They further proved that the manifold $M$ decomposes into some body regions $\Bb^i_j$ and some neck regions $\Nn^i_j$, which leads to a finite diffeomorphism type theorem of manifolds in $\Mm_1(4,v,D)$. The topology of the neck regions is controlled in the sense that $\Nn^i_j\cong \R\z S^3/\Ga^i_j$ for some discrete subgroup $\Ga^i_j\su O(4)$. (See \cite[Theorem 8.64]{cheeger2014regularity}).

Note that in this case if one applies the same argument as above, one may not get any bound for the mass of the 2-chain that fills $C_1$. For example, if the manifold is covered by contractible metric balls of two different sizes $r$ and $R$, where $R>>r$, then when we map the chain $C_1$ to $\Nc$, we may get a chain with 1-mass approximately $L/r$, but the 2-mass of the filling of this chain can be $F(N)L/r\cdot R^2 \thicksim F(N)LD\cdot(R/r)$, which is not bounded in terms of $v,D$.

Therefore, in this case, instead of looking for a 1-cycle $C_1$ in $\Si$, we are going to decompose the image of $C$ into the bodies $\Bb^i_j$ and the necks $\Nn^i_j$. We will show that one can find a decomposition $C=\sum_{k=1}^N C_k$ such that each 1-cycle $C_k$ either bounds a 2-chain in $\Bb^i_j$ or it is homologous to a cycle $C_k'$ contained in the consecutive neck region $\Nn^i_j$. In the first case, we will use the same argument as above to estimate the 2-mass of a 2-chain that fills $C_k$,  since the regions $\Bb^i_j$ can be covered by some contractible metric balls of definite size. And in the later case, using the geometry of the region $\Nn^i_j$, we show that both $C_k-C_k'$ and $C_k'$ bound some 2-chains in $M$ with controlled 2-mass.

\section{Preliminaries}
\subsection{Finite diffeomorphism type theorem and the bubble tree decomposition}
Let us first introduce the notion of harmonic radius.

\begin{definition}\label{def2_h}
Let $M^n$ be an $n-$dimensional Riemannian manifold. We define the harmonic radius $r_h(x)$ at $x\in M$ to be the largest $r>0$ such that there exists a map $\F: B_r(0^n)\ra M$, where $0^n\in \R^n$ is the origin, that satisfies the following:
\begin{enumerate}
\item $\F$ is a diffeomorphism onto its image with $\F(0^n)=x$.
\item $\De_g x^l=0$, $l=1,\dots,n$, where $x^l$ are the coordinate functions and $\De_g$ is the Laplace-Beltrami operator.
\item\label{equ1} If $g_{ij}=\F^*(g)$ is the pullback metric on $B_r(0^n)$ and $\de_{ij}$ is the standard Euclidian metric, then
$$
||g_{ij}-\de_{ij}||_{C^0,B_r(0^n)}+r||\dl_k g_{ij}||_{C^0,B_r(0^n)}\leq 10^{-3}.
$$
\end{enumerate}
\end{definition}

The above map $\F$ is called a harmonic coordinate. Note that the third condition above essentially tells us that the map $\F$ is a Lipschitz map in the following sense.

\begin{lemma}\label{lem1_c}
Let $M$ be a Riemannian manifold and $x\in M$. Suppose $r_h(x)>0$ is the harmonic radius at $x$ in Definition~\ref{def2_h} and $\F: B_{r_h(x)}(0^n)\ra M$ the harmonic coordinate. For any curve $\ga:[0,1]\ra B_{r_{h}(x)}$ whose image is contained in $B_{r_{h}(x)}$, the length of the curve $\F^{-1}(\ga)$ satisfies
$$\sqrt{1-10^{-3}}\cdot\length(\ga) \leq \length(\F^{-1}(\ga))\leq\sqrt{1+2\cdot10^{-3}}\cdot\length(\ga).$$
Moreover, we have $B_{r_{h}/2}(\Phi(0^n)) \subset \Phi(B_{r_{h}}(0^n))$.
\end{lemma}

\begin{proof}
The proof follows from direct computation. Note that by condition (3) above, the metric $\F^*(g)$ on $B_r(0^n)$ satisfies
$$
||g_{ij}-\de_{ij}||_{C^0,B_r(0^n)}+r||\dl_k g_{ij}||_{C^0,B_r(0^n)}\leq 10^{-3}.
$$
If $\ga$ is a curve realizing distance from $\partial\Phi(B_{r_{h}}(0^n)) $ to $\Phi(0^n))$, then $\Phi^{-1}(\ga)$ is a curve connecting $\partial B_{r_{h}}(0^n)$ with $0^n$. Therefore,
$$\length(\ga) \geq \length(\F^{-1}(\ga))/\sqrt{1+2\cdot10^{-3}} \geq r_h/\sqrt{1+2\cdot10^{-3}} \geq r_h/2.$$ The conclusion follows.
\end{proof}

In \cite{cheeger2014regularity}, Cheeger and Naber proved the finiteness of the number of diffeomorphism type of manifolds $M$ of dimension $4$ with $|Ric_M|\leq3$, $\vol(M)>v>0$ and $\diam(M)<D$. This theorem is based on the construction of the ``bubble tree'' decomposition of the manifolds (\cite[Theorem 8.64]{cheeger2014regularity}). Their construction is the following.

Up to rescaling, we first cover the manifold $M$ by metric balls $\{B_1(x_i)\}$ such that the balls in $\{B_{1/4}(x_i)\}$ are pairwise disjoint. By a standard volume comparison argument, there are at most $N_0(v,D)$ such balls. In each ball $B_1(x_i)$, there exist scales $r_j^1>r_0(v,D)$, an integer $N_1\leq N(v,D)$ and a collection of balls $\{B_{r_j^1}(x_j^1)\}_{j=1}^{N_1}$ such that
if $x\in B_1(x_i)\setminus \cup_j B_{r_j^1}(x_j^1)$, then the harmonic radius $r_h(x)\geq r_0(v,D)$, for some function $r_0$ and $N$ which only depends on $v$ and $D$. Furthermore, the balls $\{ B_{2r_j^1}(x_j^1)\}$ are disjoint. In total, there are at most $N_0\x N$ many such balls. We define the first body $\Bb^1=M\setminus \cup_j B_{r_j^1}(x_j^1)$ and the necks $\Nn_j^2=B_{2r_j^1}(x_j^1)\cap \Bb^1$. Note that the manifold $M=\Bb^1\cup (\cup_j B_{2r_j^1}(x_j^1))$. As proved in Theorem~8.6 and Lemma~8.40 in \cite{cheeger2014regularity}, the geometry of these $\Nn_j^2$ are controlled. Namely, there is a diffeomorphism $F: A_{r_j^1/2,2r_j^1}(0)\ra \Nn_j^2$, where $A_{r_j^1/2,2r_j^1}(0)$ is an annulus centered at $0\in \R^4/\Ga_j^2$ for some finite discrete subgroup $\Ga_j^2\su O(4)$ such that if $g_{ij}=F^*g$ is the pullback metric,
\begin{equation}\label{equ2}
||g_{ij}-\de_{ij}||_{C^0}+r_j^1\cdot|| \dl g_{ij}||_{C^0}\leq \ep(v).
\end{equation}
The order of $|\Ga_j^2|$ is bounded by a function $C(v,D)$.

We repeat this construction to each ball $B_{2r_j^1}(x_j^1)$ and all subsequent metric balls. In general, we have the bodies $\Bb_i^{k+1}=B_{2r_i^k}(x_j^k)\setminus \cup_j B_{r_j^{k+1}}(x_j^{k+1})$ and the necks $\Nn_j^{k+1}\subseteq B_{2r_j^{k}}(x_j^{k})\cap\Bb_i^{k}$ such that when $x\in \Bb_i^{k+1}$, then $r_h(x)\geq r_0\x \diam(\Bb_i^{k+1})$ and $\Nn_j^{k+1}\cong\R\z S^3/\Ga_j^{k+1}$, for some $\Ga_j^{k+1}\su O(4)$. The reason why this construction ends in finitely many steps is because if for some indices $j,k,l$, the intersection $\Nn_l^{k+1}\cap \Bb_j^k\neq \emptyset$, then $|\Ga_l^{k+1}|\leq|\Ga_j^k|-1$. Therefore, after at most $|\Ga_j^2|\leq C(v,D)$ many steps, there are no balls in the last bodies. In summary, we have the following decomposition theorem.

\begin{theorem}[\cite{cheeger2014regularity}, Theorem 8.64]\label{thm_f}
Let $M$ be a 4-dimensional Riemannian manifold with $|Ric|\leq 3$, $\vol(M)>v>0$ and $\diam(M)\leq D$. Then $M$ admits a decomposition into bodies and necks
$$M=\Bb^1 \cup \bigcup_{j_2=1}^{N_2} \Nn_{j_2}^2\cup \bigcup_{j_2=1}^{N_2} \Bb_{j_2}^2\cup\dots\cup \bigcup_{j_k=1}^{N_k} \Nn_{j_k}^k\cup \bigcup_{j_k=1}^{N_k} \Bb_{j_k}^k,$$

such that the following conditions are satisfied:
\begin{enumerate}
\item If $x\in \Bb_i^j$, then $r_h(x)\geq r_0(v,D)\cdot \diam (\Bb_i^j)$, where $r_h$ is the harmonic radius and $r_0$ is a constant that only depends on $v$ and $D$.
\item Each $\Nn_i^j$ is diffeomorphic to $\R\z S^3/\Ga_i^j$ for some $\Ga_i^j\su O(4)$.
\item $\Nn_i^j\cap \Bb_i^j$ is diffeomorphic to $\R\z S^3/\Ga_i^j$.
\item $\Nn_i^j\cap \Bb_{i'}^{j-1}$ is either empty or diffeomorphic to $\R\z S^3/\Ga_i^j$.
\item Each $N_i\leq N(v,D)$ and $k\leq k(v,D)$.
\end{enumerate}
\end{theorem}

\begin{remark}
The diffeomorphism finiteness result requires some further argument based on the decomposition in Theorem~\ref{thm_f}. One may refer to the proof of Theorem~1.12 in \cite{cheeger2014regularity} for details. The above decomposition theorem is sufficient for us to obtain the estimations in Section~3.
\end{remark}

\begin{remark}\label{rm1}
From above construction, one can see that the manifold actually decomposes as $M=\Bb^1\cup (\cup B_{2r_{j_2}^2}(x_{j_2}^2))\cup\dots\cup(\cup B_{2r_{j_k}^k}(x_{j_k}^k))=\Bb^1\cup (\cup \Bb_{j_2}^2)\cup\dots\cup(\cup \Bb_{j_k}^k)$. The above decomposition is described as the ``bubble tree'' decomposition in \cite{cheeger2014regularity}. In fact, by Theorem~\ref{thm_f}(4), one can define a tree structure from this decomposition. It is convenient to introduce the following notations for our proof in Section~3.
\end{remark}

\begin{definition}\label{def1}
Suppose the manifold $M$ has a decomposition $M=\Bb^1\cup (\cup \Bb_{j_2}^2)\cup\dots\cup(\cup \Bb_{j_k}^k)$ as in Theorem~\ref{thm_f} and Remark~\ref{rm1}. We define a tree structure so that the vertices correspond to the bodies $\Bb^j_i$. If for some indices $j$, $i$ and $i'$, the intersection $\Bb^j_i\cap\Bb^{j+1}_{i'}$ is nonempty, we connect $\Bb^{j+1}_{i'}$ and $\Bb^j_i$ with an edge. We denote by $L^j_i$ the set of children of $\Bb^j_i$, i.e.,
$$L^j_i=\{\Bb^{j+1}_k|\Bb^{j+1}_k\cap \Bb^{j}_i\neq \emptyset\}$$ and let
 $$\Ll^j_i=\cup_{\Bb^{j+1}_k\in L^j_i} \Bb^{j+1}_k.$$
We denote by $\Tt^j_i$ the set of vertices of the subtree with root $\Bb^j_i$. In other words,
$$\Tt^j_i=\Bb^j_i\cup \Ll^j_i\cup (\cup_{\Bb^{j_1}_{i_1}\in L^j_i}(\Ll^{j_1}_{i_1}\cup(\cup_{\Bb^{j_2}_{i_2}\in L^{j_1}_{i_1}}\Ll^{j_2}_{i_2}\cup\dots))).$$
\end{definition}

Let us end this section by showing some elementary result about the geometry of a neck $\Nn_i^{j+1}$ in the above decomposition theorem.
\begin{definition}
Consider a diffeomorphism $F: A_{r_i^j/2,2r_i^j}(0)\ra \Nn_i^{j+1}$ such that equation~\eqref{equ2} is satisfied, where $A_{r_i^j/2,2r_i^j}(0)$ is an annulus centered at $0\in \R^4/\Ga_i^{j+1}$ for some finite discrete subgroup $\Ga_i^{j+1}\su O(4)$. The boundary of $\Nn_i^{j+1}$ consists of 2 connected components $S_1$ and $S_2$, which corresponds to the image of the inner and outer boundary of $A_{r_i^j/2,2r_i^j}(0)$ under the map $F$. We define the thickness of the neck $\Nn_i^{j+1}$ to be
$$\thick(\Nn_i^{j+1})=\inf_{x_1\in S_1, x_2 \in S_2} \dist(x_1,x_2),$$
where $\dist(x_1,x_2)$ is the distance between $x_1$ and $x_2$ in the neck $\Nn_i^{j+1}$.
\end{definition}
Then
\begin{lemma}\label{lm_0}
$$\diam(\Nn_i^{j+1}) \leq B(v)\cdot \thick(\Nn_i^{j+1}) ,$$
where $B(v)$ is a function that only depends on $v$, and $\diam(\Nn_i^{j+1})$ is the diameter of a neck $\Nn_i^{j+1}$ as a manifold with pullback metric.
\end{lemma}
\begin{proof}
Let $F: A_{r_i^j/2,2r_i^j}(0)\ra \Nn_i^{j+1}$ be a diffeomorphism, where $A_{r_i^j/2,2r_i^j}(0)$ is an annulus centered at $0\in \R^4/\Ga_i^2$ for some finite discrete subgroup $\Ga_i^{j+1} \su O(4)$, such that if $g_{ij}=F^*g$ is the pullback metric, $||g_{ij}-\de_{ij}||_{C^0}+r_j^1|| \dl g_{ij}||_{C^0}\leq \ep(v)$. Because $\diam(S^3/\Ga_i^{j+1}) \leq \pi$, we have
\[\diam(A_{r_i^j/2,2r_i^j}(0)) \leq 2\pi r_i^j+2(2r_i^j-r_i^j/2) \leq 10 r_i^j.\]
Now $||g_{ij}-\de_{ij}||_{C^0} \leq \ep(v)$  implies that
\[\thick(\Nn_i^{j+1}) \geq \frac{2r_i^j-r_i^j/2}{\sqrt{1+2 \ep(v)}}=\frac{3r_i^j}{2\sqrt{1+2 \ep(v)}}\]
and
\[\diam(A_{r_i^j/2,2r_i^j}(0)) \geq \frac{\diam(\Nn_i^{j+1})}{\sqrt{1-\ep(v)}}.\]
Therefore, we obtain that
\[\diam(\Nn_i^{j+1}) \leq B(v)\cdot \thick(\Nn_i^{j+1}) \]
where $B(v)=\frac{20}{3}\sqrt{1-\ep(v)}\sqrt{1+2\ep(v)}$.
\end{proof}

\subsection{Nerve of a covering and geodesic graph}
Let $M\in\Mm_1(4,D,v)$. In this section, we are going to introduce the nerve of a harmonic covering of $M$ as well as a geodesic graph associated to this covering.

Let $\Bb_i^j$ be a body in the bubble tree decomposition of $M$ as in Theorem~\ref{thm_f}. The harmonic radius of a body is defined to be the infimum over harmonic radii of all points in the body. If ${r_h}_i^j$ is the harmonic radius of body $\Bb_i^j$, let
$$R_i^j={r_h}_i^j/20.$$

By Theorem~\ref{thm_f} and Remark~\ref{rm1}, one can cover each body $\Bb_i^j$ uniformly by no more than $ {(\frac{constant}{r_0(v,D)})}^4$ balls of radius $R_i^j$. We denote $\mathcal{O}(B_i^j)$ the covering of $\Bb_i^j$ by those (open) balls of radius $R_i^j$. We denote $\mathcal{O}(\Tt_i^j)=\cup_{i,j}\mathcal{O}(B_i^j)$, where the union is over all bodies in $\Tt_i^j$. Let $\mathcal{O}=\cup_{i,j}\mathcal{O}(\Tt_i^j)$. Let $\ti{N}$ be the number of balls in $\mathcal{O}$, then
$$\ti{N} \leq \left(\frac{constant}{r_0(v,D)}\right)^4 \cdot N(v,D) \cdot k(v,D).$$
We enumerate the centers of those balls as $\{x_1, \cdots, x_{\ti{N}}\}$. For simplicity, if we do not want to emphasize which body $x_i$ belongs to, we denote the harmonic ball centered at $x_i$ with radius $R_i^j$ for some $j$ by $B_{R_i}(x_i)$.
%Note that based on (1) in Theorem \ref{thm_f}, if $\Bb_{i'}^{j'}$ is a child of $\Bb_{i}^{j}$, then ${r_h}_{i'}^{j'}<{r_h}_i^j$ and hence $R_{i'}^{j'}<R_{i}^{j}$.

Let us consider the nerve $\Nc$ of the covering $\mathcal{O}$. Let the continuous map $f:M \rightarrow \Nc$ be the natural map associated to the nerve. (The reader may refer to \cite[Section 4.3]{coornaert2015topological} for the definition of the nerve and the natural map). $\Nc$ is a simplicial complex with dimension bounded by $\ti{N}$. Let $\Nc^{(k)}$ be the $k$-skeleton of $\Nc$. Note that since each $\Tt_i^j$ is covered by $\mathcal{O}(\Tt_i^j)$, let $\Nc(\Tt_i^j)$ be the nerve of $\Tt_i^j$ corresponding to this subcover. Then $\Nc(\Tt_i^j)$ is a simplicial subcomplex of $\Nc$. The natural map from $\mathcal{O}(\Tt_i^j)$ to $\Nc(\Tt_i^j)$ is defined by the restriction of $f$ to the image of $\mathcal{O}(\Tt_i^j)$ under the inclusion map.

Next, we define a geodesic graph $\Ga$ associated to the covering $\mathcal{O}$ of $M$. The vertices of $\Ga$ are the centers $\{x_1, \cdots, x_{\ti{N}}\}$ of the balls in the covering $\mathcal{O}$. If the intersection of two balls $B_{R_i}(x_i)\cap B_{R_j}(x_j) \neq \emptyset$, we connect the centers $x_i$, $x_j$ by a minimizing geodesic $E_{ij}$, and we define it to be an edge connecting $x_i$ and $x_j$ in $\Ga$. Hence, the total number of edges in $\Ga$ is bounded by $\ti{N}^2$.

Note that defined in this way, we may view $\Ga$ as a subset of $M$ using the inclusion map. And we may restrict the natural map $f:\Ga\ra \Nc$ to the geodesic graph $\Ga$.

\begin{definition}
A 1-chain $\alpha$ in $\Ga$ is a simplicial map $\alpha : [0,1]_{\Delta} \ra \Ga$, where $[0,1]_{\Delta}$ is a simplicial complex obtained by taking a partition $0 < t_1 < \cdots < t_L < 1$ of the interval $[0,1]$ and $L \geq 0$ is an integer. We define the simplicial length $m(\alpha)$ of $\alpha$ to be the number of geodesic segments in $\alpha$ counting with multiplicity. In other words, $m(\alpha) = L + 1$.
\end{definition}

Here, we describe a method to estimate the 2-mass of the filling of a 1-cycle in $\Ga$ by using the nerve $\Nc$ of the covering $\mathcal{O}$. Suppose $C$ is a 1-cycle in $\Ga$ with simplicial length $L$. We also assume that $C$ bounds a 2-chain in $M$. Then $f(C)$ is a singular 1-cycle and bounds a singular 2-chain in $\Nc$.

Suppose that $x_i$ is a vertex that is contained in the image of $C$. Then $f(x_i) \in \Nc^{(0)}$, but $f(C)$ may not be in $\Nc^{(1)}$. In fact, if $E_{ij} \subset C$ is a minimizing geodesic connecting $x_i$ and $x_j$, then $B_{R_i}(x_i)\cap B_{R_j}(x_j) \neq \emptyset$. The triangle inequality implies that $E_{ij} \subset B_{R_i}(x_i)\cup B_{R_j}(x_j)$. Hence, the image of $f(E_{ij})$ is a curve in $\Nc$ connecting $f(x_i)$ and $f(x_j)$, and $f(E_{ij})$ is contained in a simplex in $\Nc$ with vertices $f(x_i)$ and $f(x_j)$. Therefore $f(C)$ is homologous to a simplicial 1-cycle $C'$ in $\Nc$. This implies $C'$ bounds a singular 2-chain in $\Nc$ and hence it bounds a simplicial 2-chain in $\Nc$. Moreover, $C'$ has simplicial length $L$ and has the same vertices as $f(C)$.

An important observation in \cite{nabutovsky2003upper} is that if a simplicial $k$-cycle $c_k$ in an $n$-dimensional simpicial complex with $n_0$ vertices bounds some simplical $(k+1)$-chain, then it bounds a simplicial $(k+1)$-chain of which the size can be estimated using $c_k$, $n_0$ and $n$. More precisely,

\begin{lemma}[\cite{nabutovsky2003upper}, Corollary4.2] Let $C_k=\sum_i a_k^i \si_i^k $ be a simplicial $k$-cycle in an $n$-dimensional simplicial complex with $n_0$ vertices,  where $a_k^i\in \Z$ and each $\si_i^k$ being a $k$-simplex. Suppose that $C_k$ bounds a $(k+1)$-chain. Then $C_k$ bounds a simplicial $(k+1)$-chain $C_{k+1}=\sum_i a_{k+1}^i\si_i^{k+1}$ with $\sum_i |a_{k+1}^i|\leq n_0^{4n\cdot n_0^n}\max_i|a^i_k|$.
\end{lemma}

Therefore, by the above discussion of the nerve $\Nc$, we conclude that
\begin{corollary}\label{cor1}
Let $C_1=\sum_i a_1^i \si_i^1$ be a simplcial  1-cycle in $\Nc$, where $a_1^i\in \Z$ and $\si_i^1:S^1\ra M$ a Lipschitz map. Then $C_1$ bounds a 2-chain $C_2=\sum_i a_{2}^i\si_i^{2}$ in $\Nc$ with $\sum_i |a_{2}^i| \leq \ti{N}^{ 4\ti{N}^{\ti{N}+1} }\cdot \max_i|a^i_1|$.
\end{corollary}

Above Corollary~\ref{cor1} implies that $C'$ we constructed before bounds a simplicial 2-chain $C'_2$ in $\Nc$ with no more than $L\cdot\ti{N}^{ 4\ti{N}^{\ti{N}+1}}$ 2-simplices.

Our next step is to find a corresponding 2-chain in $M$ with boundary $C$. In fact, let us begin a map from the 1-skeleton $\Nc^{(1)}$ to $M$ in the following way. If $f(x_i)$ is a 0-simplex in $C'_2$, then we map it to $x_i$. A 1-simplex in $C'_2$ connecting $f(x_i)$  and $f(x_j)$ is mapped to the minimizing geodesic in $\Ga$ connecting $x_i$ and $x_j$. Therefore, the boundary of a 2-simplex in $C'_2$ is mapped to geodesic triangles in $\Ga$. The number of geodesic triangles in the image is still bounded by $L\cdot \ti{N}^{ 4\ti{N}^{\ti{N}+1}}$.

Now, if we can fill each geodesic triangle by a 2-chain with controlled 2-mass, then we may obtain a filling of $C$ with bounded 2-mass. Note that a geodesic triangle in $\Ga$ is contained in a harmonic ball. Indeed, by the triangle inequality, if $x_i$, $x_j$ and $x_k$ are vertices of a geodesic triangle, assume that the harmonic radius at $x_i$ is largest among the three points, then the geodesic triangle is contained in $B_{R_i}(x_i)\cup B_{R_j}(x_j) \cup B_{R_k}(x_k) \subset B_{r_h/2}(x_i)$. We show that if the image of a 1-cycle is contained in a harmonic ball, then it bounds a ``small'' 2-chain inside the ball.

\begin{lemma}\label{lm1}
Suppose $x\in M$ is a point in $M$. Let $r_h(x)$ be the harmonic radius at $x$. Then any singular Lipschitz 1-cycle $C\in \Zz_1(M,\Z)$ whose image of $C$ is contained in the metric ball $B_{r_h(x)/2}(x)$ with $\mass_1 (C)$ bounds a singular 2-chain with 2-mass bounded by $\mass_1 (C) \cdot r_h(x)$.
\end{lemma}
\begin{proof} Let $\F:B_{r_h(x)}(0)\ra M$ be the harmonic coordinate at the point $x\in M$, where $\F(0)=x$. Suppose that $C=\sum_i a_if_i$, where $a_i\in \Z$ and $f_i:S^1\ra M$ being Lipschitz maps. Let $l_i=\mass_1(f_i)$. We show that $f_i$ bounds a 2-chain in $B_{r_h(x)/2}(x)$ with mass bounded by $l_i\cdot r_h(x)$.

We use the same notation $f_i$ to denote the image of the 1-cycle $f_i$ in the metric ball $B_{r_h(x)/2}(x)$. By Lemma~\ref{lem1_c}, $\F^{-1}(f_i)$ is a piecewise differentiable curve in $B_{r_h(x)}(0)\su \R^4$ with length bounded by $\sqrt{1+2\cdot 10^{-3}}\cdot l$. Let us fill $\F^{-1}(f_i)$ by a singular 2-chain $B_i$ in $B_{r_h(x)}(0)$, which is obtained by connecting the points on $\F^{-1}(f_i)$ with the origin in $\R^4$. This 2-chain is a generalized cone with boundary $\F^{-1}(f_i)$. The area of this cone is bounded by $\frac{1}{2}\x r_h(x)\x \length(\F^{-1}(f_i))\leq \frac{\sqrt{1+2\cdot 10^{-3}}}{2}\x l\x r_h(x)$. Then $\F(B_i)$ defines a 2-chain in $B_{r_h(x)}(x)$ with boundary $f_i$. By Lemma~\ref{lem1_c},  $\mass_2(\F(B_i))\leq \frac{\sqrt{1+2\x 10^{-3}}}{2\x(1-10^{-3})} \x l\x r_h(x)\leq l\x r_h(x).$
\end{proof}

\section{Proof of Theorem~\ref{thm1}}
Throughout this section, we consider the manifold $M\in \Mm_1(4,D,v)$. Recall that the idea of the proof of this theorem is the following. We begin with a singular Lipschitz 1-cycle $C\in \Zz_1(M;\Z)$. Our first step is to find a ``nice'' representative of the cycle whose image is contained in some body regions. In other words, we are going to show that the cycle $C$ can be written as a sum $\sum C_k$, where the image of each $C_k$ is contained in some body $\Bb_i^j$. Then for each $C_k$, we are going to show that it either bounds a 2-chain in $\Bb_i^j$ with controlled mass, or it is homologous to a 1-cycle $C_k'$ of which the image is contained in the neck $\Nn_{i}^{j+1}$. We will show that in the later case the cycle $C_k-C_k'$ bounds a 2-chain in $\Bb_i^j$ and the cycle $C_k'$ bounds a 2-chain in $M$. The mass of both these 2-chains can be estimated using the results in Section~2.2. Hence the 1-cycle $C$ bounds a 2-chain with controlled mass.

First we prove that every cycle can be represented by maps with image contained in a single body.

\begin{lemma}\label{lm2}
For any singular Lipschitz 1-cycle $C\in \Zz_1(M,\Z)$, there is a decomposition $C=\sum_{k=1}^{n(v,D)} C_k$ such that
\begin{enumerate}
\item Each $C_k$ is a singular Lipschitz 1-cycle with image contained in a body $\Bb^j_i$ for some $\Bb^j_i$ in Theorem~\ref{thm_f}.
\item The total number of the cycles $n(v,D) \leq N(v,D)\cdot k(v,D)$, where $N(v,D)$ and $k(v,D)$ are the functions in Theorem~\ref{thm_f}.
\item $\sum_{k=1}^{n} \mass_1(C_k)\leq (2B(v)+1)^{k(v,D)}\cdot \mass_1(C)$, where $B(v)$ is the function in Lemma~\ref{lm_0}.
\end{enumerate}
\end{lemma}
\begin{proof} Suppose the 1-cycle $C=\sum_k a_k f_k$, where each $f_k:S^1\ra M$ is a Lipschitz map. We consider the image of each map $f_k$ respectively. We are going to decompose the image of each $f_k$ into bodies $\Bb_i^j$. In other words, we are going to represent the cycle $f_k$ as a sum $\sum_{m=1}^{N_k}C_{m}^k$ by preforming some construction on the image of $f_k$ in $M$, where each $C_{m}^k$ is a Lipschitz 1-cycles whose image is contained in some $\Bb^j_i$ and $N_k$ is a positive integer.

We are going to construct this decomposition by induction. Recall that $\Tt^j_i$ is the union of some bodies in Definition~\ref{def1}. Let $j$ be the largest integer such that the image of $f_k$ is contained in $\Tt^j_i$, for some $i$. We will represent $f_k$ by $C_j^k+\sum_m C^k_{j+1,m}$ such that the image of $C_j^k$ is contained in $\Bb^j_i$, and the image of each $C^k_{j+1,m}$ is contained in $\Tt_{i_m}^{j+1}$.

\begin{figure}[htbp]
\begin{minipage}[c]{0.47\textwidth}
\centering\includegraphics[width=7cm]{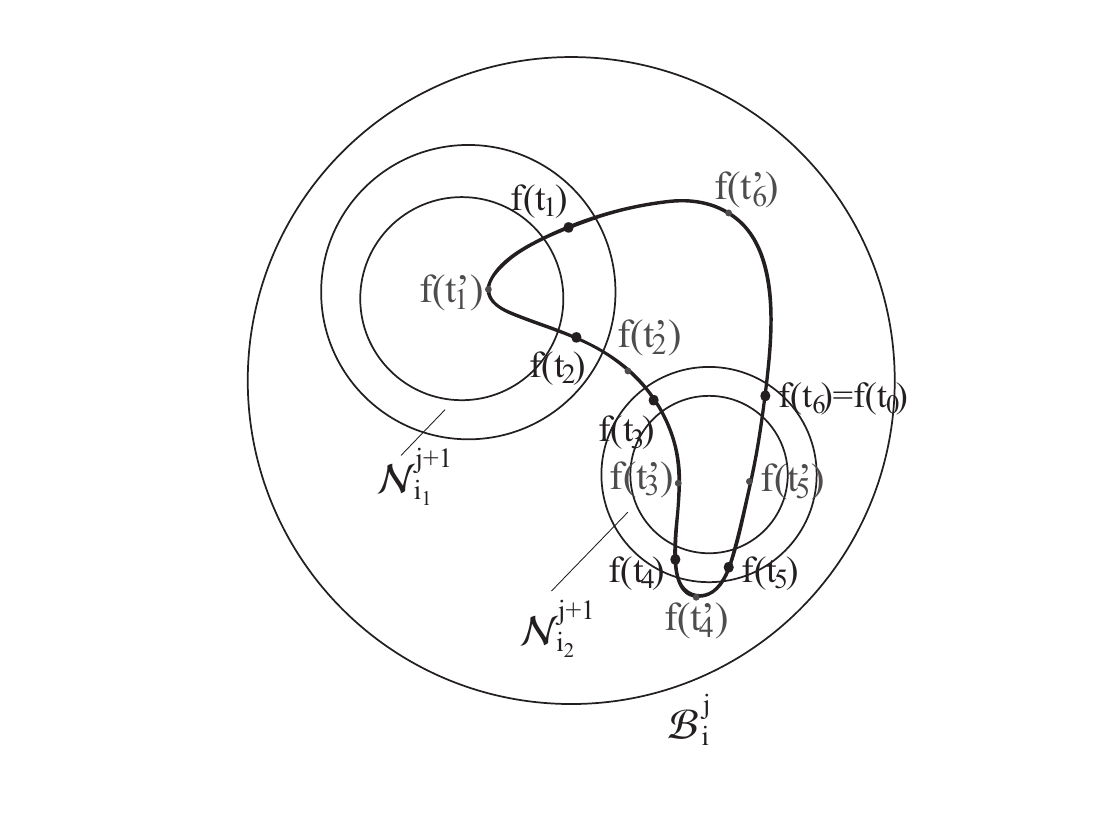}
\vspace*{6pt}
\caption{Choosing the points $f(t_k)$ and $f(t'_k)$}\label{fig1}
\end{minipage}
\begin{minipage}[c]{0.47\textwidth}
\centering\includegraphics[width=7cm]{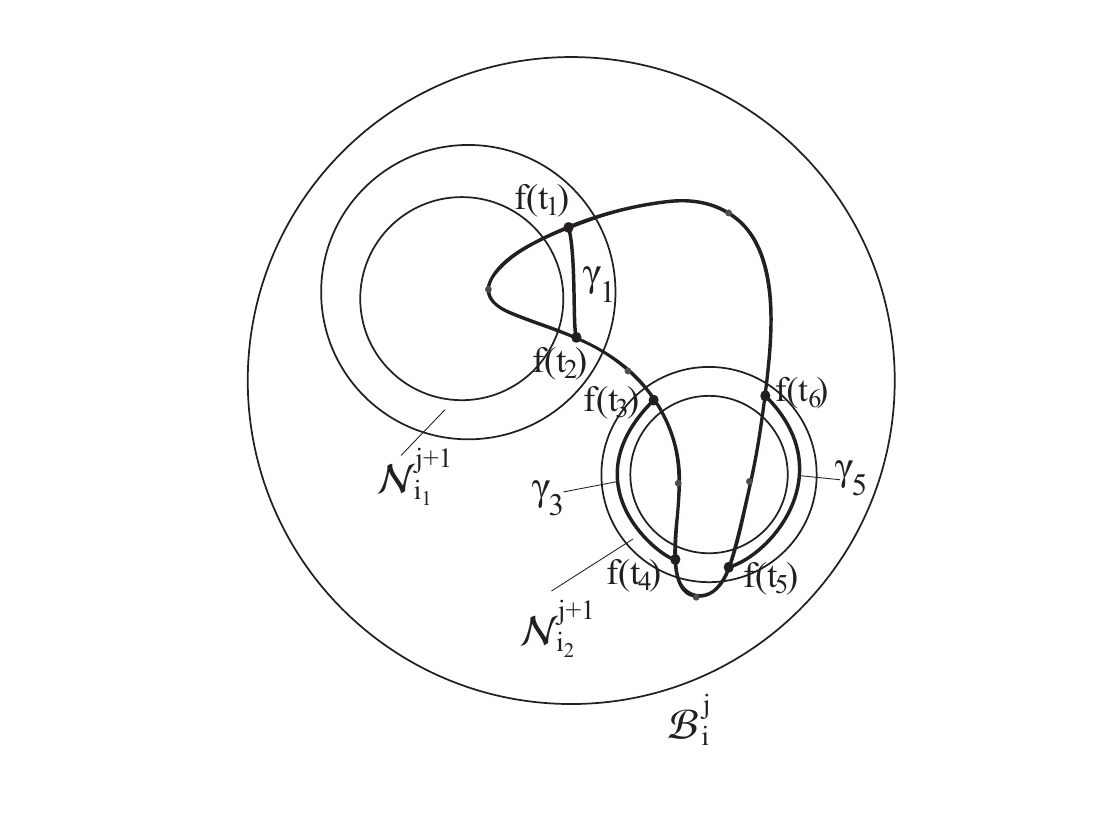}
\vspace*{6pt}
\caption{Connecting $f_k(t_s)$ and $f_k(t_{s+1})$ by a minimizing geodesic $\ga_s$}\label{fig2}
\end{minipage}
\end{figure}

Let us identify the domain $S^1$ of $f_k$ with the interval $[0,1]$ and parameterize $f_k$ by $f_k:[0,1]\ra M$ such that $f_k(0)=f_k(1)$. We first take a partition $0=t_{n+1}=t_0<t_1<\dots<t_n=1$ of $[0,1]$ such that:
\begin{enumerate}
\item Each point $f_k(t_s)$ is contained in some neck $\Nn^{j+1}_{i_s}$.
\item The arcs $f_k([t_s,t_{s+1}])\su \Bb^j_i$ for $s$ even and $f_k([t_s,t_{s+1}])\su \Tt^{j+1}_{i_s}$ for $s$ odd.
\end{enumerate}

We then take a refinement of the above partition. For any integer $s$, we pick $t_s'$ with $t_s<t_s'<t_{s+1}$ such that $f_k(t_s')\in \Bb^j_i- \Tt^{j+1}_{i_s}$ for $s$ even and $f_k(t_s')\in \Tt^{j+1}_{i_s}- \Bb^{j}_i$ for $s$ odd. (See Figure~\ref{fig1}.)

When $s$ is odd, suppose that $f_k([t_s,t_{s+1}])\su T^{j+1}_{i_s}$. Because $\Nn^{j+1}_{i_s}$ is connected, both $f_k(t_s)$ and $f_k(t_{s+1})$ are contained in $\Nn^{j+1}_{i_s}$. In this case, we connect $f_k(t_s)$ and $f_k(t_{s+1})$ by a minimizing geodesic $\ga_s$ in $\Nn^{j+1}_s$. (See Figure~\ref{fig2}.)

By Lemma \ref{lm_0}, (also see Figure~\ref{fig3}), we have
$$\length(\ga_s)\leq \diam(\Nn^{j+1}_{i_s})\leq B(v)\x \thick(\Nn^{j+1}_{i_s})\leq B(v)\x \length(f_k([t_s',t_{s+1}'])).$$
Therefore, we have an estimation of the sum of the length of curves $\ga_s$ given by
$$\sum_{s \text{ odd}}\length(\ga_s)\leq B(v)\x \sum_{s \text{ odd}} \length(f_k([t_s',t_{s+1}']))\leq B(v)\x \length(f_k).$$

When $s$ is even, by construction, the arc $f_k([t_s,t_{s+1}])\su \Bb^j_i$. Therefore, the image of the curve $f_k([t_0,t_1])\cup\ga_1\cup\dots\cup\ga_{n-1}\cup f_k([t_n,t_{n+1}])$ is contained in $\Bb^j_{i}$. We take $C^m_j$ to be a 1-cycle whose image is this curve. (See Figure~\ref{fig4}.)

When $s$ is odd, the arc $f_k([t_s,t_{s+1}])\su \Tt^{j+1}_{i_s}$ and $f_k([t_s,t_{s+1}])\cup(\ga_s)$ is a closed curve. For a fixed $i$, we take $C^m_{j+1,i}$ to be a cycle whose image is the union of all $f_k([t_l,t_{l+1}])\cup(\ga_l)$'s such that $f_k([t_l,t_{l+1}])\su \Tt_{i}^{j+1}$.

\begin{figure}[htbp]
\begin{minipage}[c]{0.47\textwidth}
\centering\includegraphics[width=7cm]{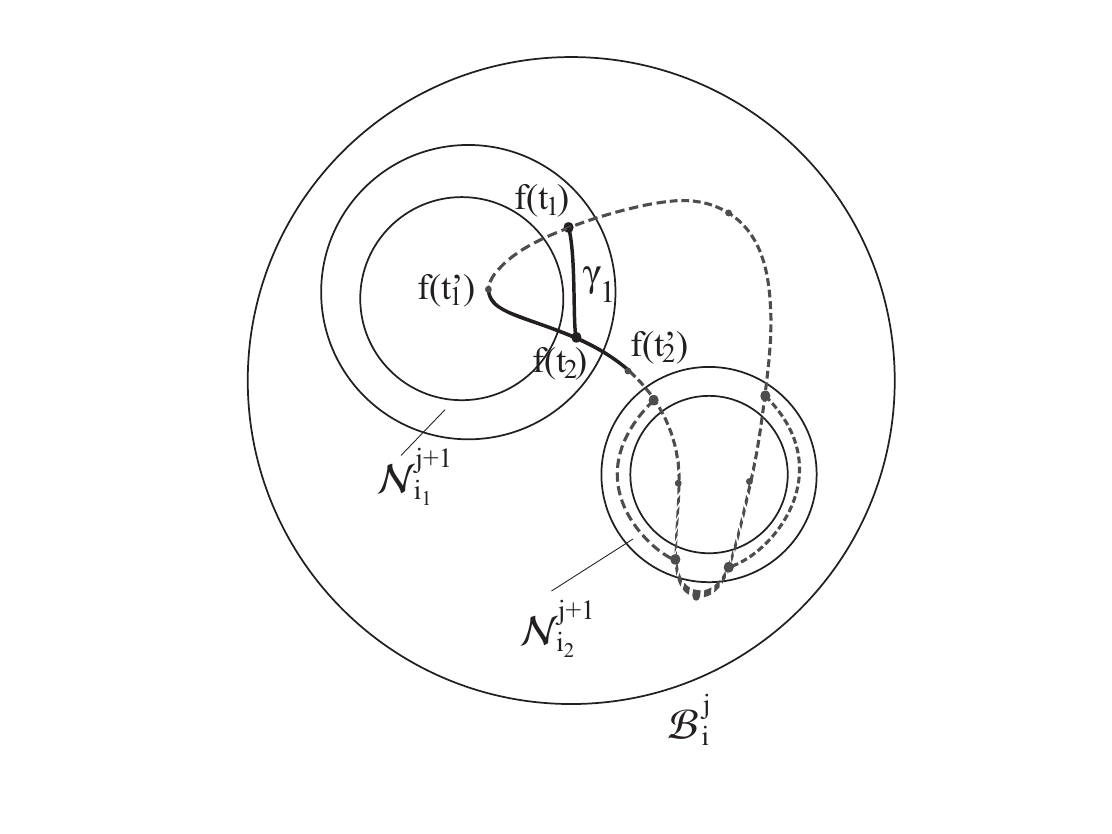}
\vspace*{6pt}
\caption{$\length(\ga_1)\leq B(v)\x \length(f([t_1',t_2'])).$}\label{fig3}
\end{minipage}
\begin{minipage}[c]{0.47\textwidth}
\centering\includegraphics[width=7cm]{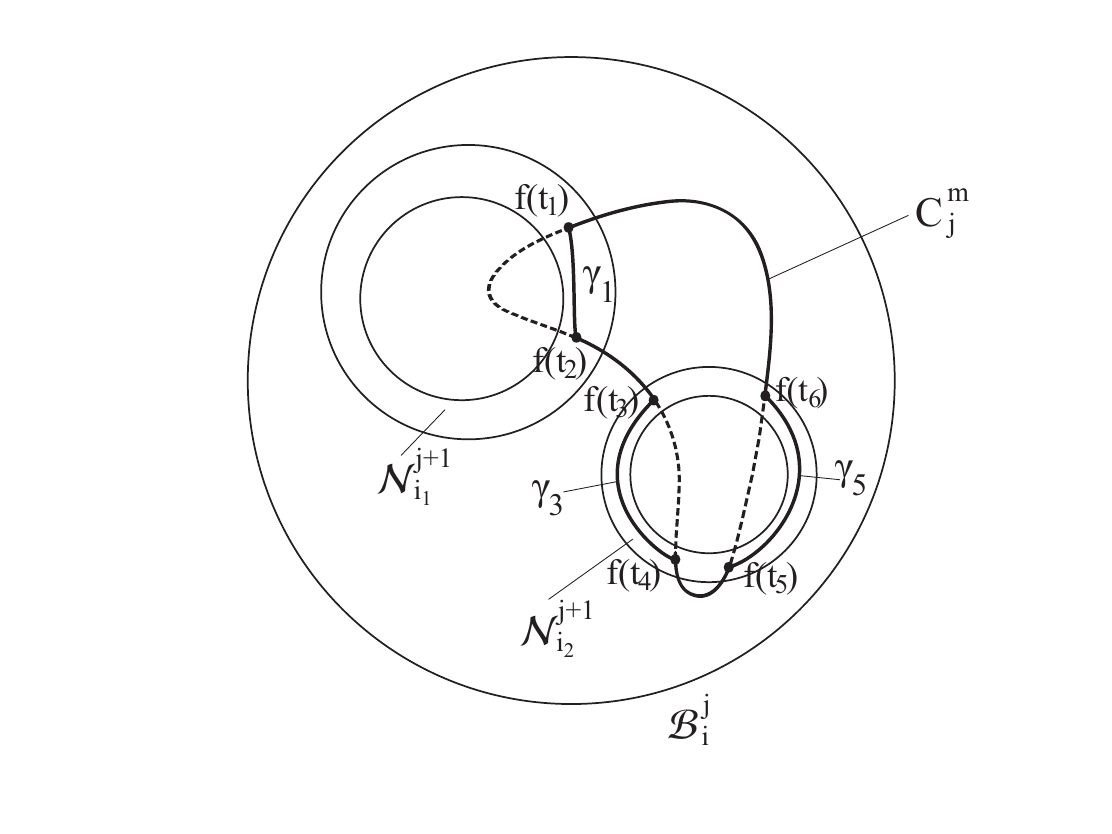}
\vspace*{6pt}
\caption{Choosing $C_j^m$.}\label{fig4}
\end{minipage}
\end{figure}

Then $f_k$ can be represented by $C^k_j+\sum_{i\in I} C^k_{j+1,i}$, for some index set $I$ and
$$\mass_1(C^k_j)+\sum_{i\in I}\mass_1(C^k_{j+1,i})\leq \length(f_k)+2\sum_{s \text{ odd}}\length(\ga_s)\leq (2B(v)+1)\x \length(f_k).$$

We then apply the same construction to each $C^k_{j+1,i}$ and to every $k$. Now for every fixed $\Bb^j_i$, let $C_{i_j}$ be the cycle whose image is the union of all $C^k_{m,l}$'s such that the image of $C^k_{m,l}\su \Bb^j_i$.

Then $C=\sum_{i_j\in J} C_{i_j} $, for some index set $J$. The size of $J$ is bounded by the total number of the bodies $n(v,D) \leq N(v,D)k(v,D)$. And the mass $$\sum_{i_j\in J} \mass_1(C_{i_j})\leq (2B(v)+1)^{k(v,D)}\mass_1(C),$$
\end{proof}

As described above, our next step is to show that for any Lipschitz 1-cycle $C$ whose image is contained in a single body $\Bb_i^j$, it either bounds a ``small'' 2-chain whose image is contained in $\Tt_i^j$, or it is homologous to some cycle in $\Nn^{j}_i$. And the idea is to apply certain finiteness property of the geodesic graph $\Ga$ we constructed in Section~2.2. To begin with, we first show that the cycle $C$ is homologous to a cycle $C'$ in subgraph of $\Ga$. The centers of this subgraph are in $B_i^j$. The mass of this cycle $C'$ is comparable with the mass of the cycle $C$.

\begin{lemma}\label{lm3}
For any singular Lipschitz 1-cycle $C$ with its image contained in the body $\Bb_i^j$, $C$ is homologous to a 1-cycle $C'$ in $\Ga$. The image of $C'$ is also contained in $\Bb_i^j$. And the simplicial length of $C'$ is bounded by $\ti{N}^2\cdot\mass_1(C)/R_i^j$, where $R_i^j={r_h}^j_i/20$ as we defined in Setcion~2.2. Furthermore,
$$mass_1(C') \leq 2\ti{N}^2\cdot\mass_1(C).$$
There is a singular 2-chain $E$ with $\dl E=C-C'$ and
$$\mass_2(E)\leq 20R_i^j\x (2\ti{N}^2+3)\cdot\mass_1(C').$$
\end{lemma}
\begin{proof} Suppose that $C=\sum a_kf_k$, where $a_k\in\Z$ and $f_k:S^1\ra M$ is Lipschitz. Let us consider each map $f_k$. We parameterize the map $f_k$ by $f_k:[0,1]\ra M$. Let us take a partition $0=t_0<t_1<\dots <t_L=1$ of the interval $[0,1]$ such that $\length(f_k(t_m,t_{m+1}))=R_i^j$. We have
$$L\leq\mass_1(f_k)/R_i^j.$$

Suppose that $f_k(t_m)\in B_{R_i^j}(x)$ and $f_k(t_{m+1})\in B_{R_i^j}(y)$ for some $x,y \in \Bb_i^j$. Denote by $\si_m$ a minimizing geodesic connecting $f_k(t_m)$ and $x$, and $\si_{m+1}$ a minimizing geodesic connecting $f_k(t_{m+1})$ and $y$. Note that by the triangle inequality, the distance between $x$ and $y$ is less than $3R_i^j$, so $B_{R_i^j}(x)\cap B_{R_i^j}(y)=\emptyset$.

Our next step is to construct a 1-chain $\al_m$ in $\Ga$ with boundary $x$ and $y$ so that the 1-cycle $\ga_m:=-\al_m\cup \si_m\cup f_k(t_m,t_{m+1})\cup (-\si_{m+1})$ is contained in $B_{{r_{h}}_i^j(x)/2}(x)$. The 1-chain $\al_m$ is constructed in following way. We partite the interval $[t_i,t_{i+1}]$ further into $t_i=t_{i,0} < t_{i,1}<\cdots <t_{i,u}=t_{i+1}$ such that if $f_k(t_{i,v}) \in B_{R_i^j}(x_v)$ and $f_k(t_{i,v+1}) \in B_{R_i^j}(x_{v+1})$, then $B_{R_i^j}(x_v) \cap B_{R_i^j}(x_{v+1}) \neq \emptyset$. Note that those $x_v$ may repeat. There is a 1-chain in $\Ga$ connecting $x$ and $y$ through centers $x_v$'s.  In other words there is a continuous map $\beta_m:[0,1] \ra \Ga$ and $0=s_0 < \cdots <s_u=1$, such that $\beta_m(0)=x$ and $\beta_m(1)=y$ and $\beta_m(s_v)=x_v$.

If $\beta_m(s_v)=\beta_m(s_{v'})$, we delete $\beta_m([s_v, s_{v'}])$. We define $\al_m$ to be the resulting a 1-chain after we deleted all the repetitions in $\be_m$. Therefore, the simplicial length of $\al_m$ is bounded by $\ti{N}^2$, and hence the 1-mass of $\al_m$ is bounded by $2R_i^j \cdot \ti{N}^2$. And therefore,
$$\mass_1(\ga_m)\leq 2R_i^j\x\ti{N}^2+3R_i^j.$$

Now by Lemma~\ref{lm1}, each $\ga_m$ bounds a 2-chain $E_m$ with 2-mass bounded by
$$\mass_2(E_m)\leq {r_h}_i^j\x R_i^j\x (2\ti{N}^2+3) \leq 20R_i^j\x(2\ti{N}^2+3)\cdot R_i^j.$$

Let $\al=\sum_m \al_m$. The simplicial length of $\al$ is bounded by $ \ti{N}^2 \cdot L\leq \ti{N}^2\cdot\mass_1(f_k)/R_i^j$ and hence the 1-mass of $\al$ is bounded by $2R_i^j \ti{N}^2 \cdot L\leq 2\ti{N}^2\cdot\mass_1(f_k)$.

Note that $\sum_m \ga_m=f_k-\al$. Therefore, we conclude that $\al$ is homologous to $f_k$ and the cycle $\sum_m \ga_m=f_k-\al$ bounds a 2-chain $\sum_m E_m$ with 2-mass bounded by $20R_i^j\x (2\ti{N}^2+3)\cdot \mass_1(\al)$. Now we apply this construction to all components $f_k$'s and let $C'$ be the sum of all corresponding $\al$'s, we conclude that $C'$ satisfies the desired properties.
\end{proof}

Now we are going to show that a cycle $C$ whose image is contained in $\Bb_i^j$ either bounds a 2-chain in $\Tt_i^j$, or it can be ``move out'' from $\Bb_i^j$. In the later case the cycle $C$ is homologous to a cycle whose image is contained in the neck region $\Nn_i^j\su \Bb_i^j\cap\Bb_{i'}^{j-1}$ for some body $\Bb_{i'}^{j-1}$.

\begin{lemma}\label{lm4}
Let $C$ be a singular Lipschitz 1-cycle with its image contained in a body $\Bb_i^j$. Then either $C$ bounds a singular 2-chain in $\Tt_i^j$ or $C$ is homologous to a singular Lipschitz 1-cycle $C''$ such that
\begin{enumerate}
\item The image of $C''$ is contained in $\Nn_i^j\su \Bb_i^j\cap \Bb^{j-1}_i$.

\item There is a singular 2-chain $E$ with its image contained in $\Tt_i^j$ such that $\dl E=C-C''$.

\item $\mass_1(C'')\leq h(v)\cdot \diam(\Nn_i^j)$, for some function $h$ which only depends on $v$ and $\diam(\Nn_i^j)$ is the diameter of a neck $\Nn_i^j$ as a manifold with pullback metric.
\end{enumerate}
\end{lemma}
\begin{proof} By our assumption, the first homology $H^1(M)=0$. Hence the 1-cycle $C$ bounds a 2-chain in $M$. Suppose the image of this 2-chain is not contained in $\Tt_i^j$. Then $C$ is homologous to a Lipschitz 1-cycle $C'$ in $\Nn_i^j$. To see this, one can take a small perturbation of the image of the 2-chain so that the intersection between the 2-chain and and boundary of $\Nn_i^j$ is transverse and we define $C'$ to a cycle whose image is given by this intersection.

Note that since $\Nn^j_i$ is diffeomorphic to $\R\z S^3/\Ga^j_i$ for some $\Ga^j_i\su O(3)$, the fundamental group $\pi_1(\Nn^j_i)\cong\Ga^j_i$ is finite. In particular, the size of the first homology group $|H_1(\Nn^j_i)|\leq |\Ga^j_i|\leq h_1(v)$, where $h_1$ is a function which only depends on $v$. By Gromov Lemma \cite{gromov2007metric}, the generators of $H_1(\Nn_i^j)$ can be represented by curves of length bounded by $2\diam(\Nn_i^j)$. Since $H_1(\Nn_i^j)$ is finite abelian group with at most $h_1(v)$ generators, one can choose a representative $C''$ in the class $[C']$ such that $\mass_1(C'')\leq 2h_1(v)^{h_1(v)}\diam(\Nn_i^j):=h(v)\diam(\Nn_i^j)$.
\end{proof}

\begin{remark}
In the above Lemma, we have no control of the size of the 2-chain $E$. In fact the 2-mass of the 2-chain $E$ can be large since we are trying to find a ``nice'' representative of the class $[C']$.
\end{remark}

We are going to show below in Lemma~\ref{lm5} and Lemma~\ref{lm6} that in both two cases of Lemma~\ref{lm4}, one can always construct a ``new'' 2-chain with boundary $C$ such that the 2-mass of the chain is bounded by a function that only depends on $v$, $D$ and the 1-mass of $C$.

\begin{lemma}\label{lm5}
Let $C$ be a singular Lipschitz 1-cycle whose image is contained in $\Bb_i^j$. If $C$ bounds a singular 2-chain whose image is contained in $\Tt_i^j$, then $C$ bounds a singular 2-chain $E$ in $M$ such that $\mass_2(E)\leq g_1(v,D)\cdot\mass_1(C)$, for some function $g_1$ that only depends on $v$ and $D$.
\end{lemma}
\begin{proof} Suppose that $C$ bounds a singular 2-chain in $\Tt_i^j$. By Lemma~\ref{lm3}, $C$ is homologous to a 1-cycle $C'$ in $\Ga$. The image of $C'$ is contained in $\Bb_i^j$ and the simplicial length of $C'$ is bounded by $2\ti{N}^2\cdot\mass_1(C)/R_i^j$. Furthermore, there is a singular 2-chain $G$ such that $\dl G=C-C'$ and
$$\mass_2(G)\leq 20R_i^j\x (2\ti{N}^2+3)\x \mass_1(C)\leq20D\x(2\ti{N}^2+3)\x \mass_1(C).$$
Because $C$ and $C'$ are homologous, the cycle $C'$ also bounds a 2-chain in $\Tt^i_j$. If we apply the method in Section~2.2 and Corollary~\ref{cor1}  to $\Nc(\Tt^i_j)$, the 1-cycle $C'$ bounds a 2-chain $F$ in $\Tt^i_j$. The the image of the boundary of $F$ consists of at most $2\ti{N}^{4\ti{N}^{\ti{N}+1} }\ti{N}^2\mass_1(C)/R_i^j$ geodesic triangles in $\Ga$.

Note that these geodesic triangles are also in $\Tt^i_j$. And if $\Bb_{i'}^{j'}$ is a child of $\Bb_{i}^{j}$, then $R_{i'}^{j'}<R_{i}^{j}$. Therefore the circumference of each geodesic triangle is at most $3\x 2R_{i}^{j}$. Hence each geodesic triangle can be contained in a harmonic ball of radius $20R_i^j$. By Lemma~\ref{lm1}, one can fill each geodesic triangle by a 2-chain with area at most $6\x 20(R_i^j)^2$. In other words, the 2-mass of $F$ is bounded by $240\ti{N}^{4\ti{N}^{\ti{N}+1}}\x\ti{N}^2\x\mass_1(C)\cdot R_i^j \leq 240 D\ti{N}^{4\ti{N}^{2\ti{N}}}\x \mass_1(C)$. Let $E=F'+G$. Then $\dl E=\dl(F'+G)=C-C'+C'=C$ and
$$\mass_2(E)\leq (240\ti{N}^{4\ti{N}^{2\ti{N}}}+40\ti{N}^2+60)D\mass_1(C).$$
\end{proof}

\begin{lemma}\label{lm6}
Let $C$ be a singular Lipschitz 1-cycle in $\Bb_i^j$. If $C$ is homologous a singular 1-cycle $C'$ in $\Nn_i^j$ as described in Lemma~\ref{lm4}, then $C$ bounds a singular 2-chain $E$ in $M$ such that $\mass_2(E)\leq g_1(v,D)\cdot\mass_1(C)+g_2(v,D)f(v)$, for some function $g_1$, $g_2$ and $f$.
\end{lemma}
\begin{proof} By Lemma~\ref{lm4}, the 1-mass of the 1-cycle satisfies $\mass_1(C')\leq h(v)\diam(\Nn_i^j)$.  Since $\Nn_i^j$ is covered by no more than $\ti{N}$ balls of radius $R_i^j$, $\diam(\Nn_i^j) \leq \ti{N}\x R_i^j \leq \ti{N}D$.

Therefore, $\mass_1(C')\leq h(v)\x \ti{N}R_i^j \leq h(v)\x\ti{N}D$. And by Lemma~\ref{lm3}, $C$ is homologous to $C''$ in $B^i_j$. $C''$ is a 1-cycle in $\Ga$ such that $\mass_1(C'')\leq 2R_i^j\ti{N}^3h(v) \leq 2D\ti{N}^3h(v) $ and its simplicial length is bounded by $\ti{N}^3h(v)$. Therefore, by Lemma~\ref{lm5}, the cycle $C-C''$ bounds a 2-chain $G$ such that $\mass_2(G)\leq g_1(v,D)(\mass_1(C)+\mass_1(C''))$.

Now because $H^1(M)$ is trivial, the cycle $C''$ bounds a 2-chain in $M$. We apply the method in Section~2.2 and Corollary~\ref{cor1} to $C''$ to conclude that the cycle $C''$ bounds a 2-chain $F$ in $M$ such that the image of the boundary of $F$ consists of at most $\ti{N}^{4\ti{N}^{\ti{N}+1} }\ti{N}^3h(v)$ geodesic triangles in $\Ga$. The length of each edge is at most $D$ and the triangle is contained in a harmonic ball of radius at most $D$. By Lemma\ref{lm1} Each triangle can be filled by a 2-chain with 2-mass bounded by $3D^2$. Therefore, we obtain a 2-chain $F$ with $\dl F=C''$ and
$$\mass_2(F)\leq 3\ti{N}^{4\ti{N}^{\ti{N}+1} }\ti{N}^3 \x D^2\x h(v).$$

Let $E=G+F$. Then $\dl E=C-C''+C''=C$ and
$$\mass_2(E)\leq g_1(v,D)\mass_1(C)+g_2(v,D)h(v),$$
where
$$ g_1(v,D)=(240\ti{N}^{4\ti{N}^{2\ti{N}}}+40\ti{N}^2+60)D,
$$
and
$$ g_2(v,D)= 3\ti{N}^{4\ti{N}^{\ti{N}+1} }\ti{N}^3 \x D^2+g_1(v,D)\x 2D\ti{N}^3.
$$
\end{proof}

We may now complete the proof of Theorem~\ref{thm1}.

\begin{proof}[Proof of Theorem~\ref{thm1}]
Let $C\in\Zz_1(M,\Z)$ be a singular Lipschitz 1-cycle. By Lemma~\ref{lm2}, we can find a representation of $C=\sum_{k=1}^n(v,D) C_k$, where $n(v,D)\leq N(v,D)\x k(v,D)$ and each $C_k$ is a Lipschitz 1-cycle with image contained in some body $\Bb_j^i$ and the total mass satisfies
$$\sum_{k=1}^n\mass_1(C_k)\leq(2B(v)+1)^{k(v,D)}\x \mass_1(C).$$

Now by Lemma~\ref{lm4}, each $C_k$ either bounds a 2-chain in $\Tt_j^i$, or it is homologous to some $C'_k$ such that the image of $C'_k$ is contained in $\Nn_j^i\su \Bb_j^i\cap\Bb_j^{i-1}$.

In the first case, Lemma~\ref{lm5} tells us that $C_k$ bounds a 2-chain $E_k$ with
$$\mass_2(E_k)\leq g_1(v,D)\mass_1(C_k).$$

And in the second case, we apply Lemma~\ref{lm6} to obtain a 2-chain $E_k$ with
$$\mass_2(E_k)\leq g_1(v,D)\mass_1(C_k)+g_2(v,D)h(v).$$
If we take the sum over all $n$, $\sum_{k=1}^n C_k$ bounds 2-chain with 2-mass bounded by
$$g_1(v,D) \x \sum_{k=1}^n\mass_1(C_k)+g_2(v,D)\x n(v,D)\x h(v).$$ Therefore we conclude that $C$ bounds 2-chain with 2-mass bounded by
 $$f_1(v,D)\mass_1(C)+f_2(v,D),$$

where
$$f_1(v,D)=g_1(v,D)\x (2B(v)+1)^{k(v,D)}$$

and
$$f_2(v,D)=g_2(v,D)\x h(v)\x N(v,D)\x k(v,D).$$
\end{proof}

\section*{Acknowledgement}
The authors are grateful to Alexander Nabutovsky and Regina Rotman for suggesting this problem and numerous helpful discussions.
We also thank Vitali Kapovitch and Robert Haslhofer for useful discussions about the epsilon-regularity theorem. We thank Aaron Naber for answering several questions about his work \cite{cheeger2014regularity} with Jeff Cheeger.

\bigskip
\bibliographystyle{alpha}
\bibliography{mybib}

\end{document}